\documentclass[11 pt, reqno]{amsart}
\usepackage{amsmath,times,epsfig,amssymb,amsbsy,amscd,amsfonts,amstext,graphicx,color,bm,amsthm}
\usepackage[all]{xy}
\usepackage{graphicx}
\usepackage{hyperref, comment}

\usepackage{tikz,subfigure}
\usepackage{tikz-cd}

\newcommand{\KK}{\mathbb{K}}

\newcommand{\RR}{\mathbb{R}}
\newcommand{\PP}{\mathbb{P}}

\newcommand{\A}{\mathcal{A}}
\newcommand{\T}{\mathcal{T}}
\newcommand{\OO}{\mathcal{O}}

\newcommand \ideal[1] {\langle #1 \rangle}

\newtheorem{Theorem}{Theorem}[section]

\newtheorem{Lemma}[Theorem]{Lemma}
\newtheorem{Proposition}[Theorem]{Proposition}
\newtheorem{Corollary}[Theorem]{Corollary}
\newtheorem{Conjecture}[Theorem]{Conjecture}

\theoremstyle{definition}

\newtheorem{Definition}[Theorem]{Definition}
\newtheorem{Remark}[Theorem]{Remark}
\newtheorem{Example}[Theorem]{Example}

\DeclareMathOperator{\POexp}{POexp}
\DeclareMathOperator{\depth}{depth}
\DeclareMathOperator{\Ass}{Ass}
\DeclareMathOperator{\rk}{rk}
\DeclareMathOperator{\Der}{Der}
\DeclareMathOperator{\projdim}{projdim}
\DeclareMathOperator{\pdeg}{pdeg}
\DeclareMathOperator{\codim}{codim}
\DeclareMathOperator{\gin}{gin}
\DeclareMathOperator{\rgin}{rgin}
\DeclareMathOperator{\LT}{LT}
\DeclareMathOperator{\GL}{GL}
\DeclareMathOperator{\HF}{HF}

\DeclareMathOperator{\Ext}{Ext}
\DeclareMathOperator{\reg}{reg}
\DeclareMathOperator{\sat}{sat}

\begin{document}

\title[Lefschetz properties and $3$-dimensional arrangements]{Lefschetz properties and the Jacobian algebra of $3$-dimensional hyperplane arrangements}

\begin{abstract}
In this article, we study the weak and strong Lefschetz of higher dimensional quotients and dimension 1 almost complete intersections. We then apply the obtained results to the study of the Jacobian algebra of hyperplane arrangements.
\end{abstract}

\author{Simone Marchesi}
\address{Simone Marchesi, Universitat de Barcelona, Facultat de Matem\`atiques i Inform\`atica, Gran Via de les Corts Catalanes 585, 08007, Barcelona, Spain.\\ 
Centre de Recerca Matem\`atica Edifici C, Campus Bellaterra, 08193 Bellaterra, Spain}
\email{marchesi@ub.edu}
\author{Elisa Palezzato}
\address{Elisa Palezzato, Department of Mathematics, Hokkaido University, Kita 10, Nishi 8, Kita-Ku, Sapporo 060-0810, Japan.}
\email{elisa.palezzato@gmail.com}
\author{Michele Torielli}
\address{Michele Torielli, Department of Mathematics \& Statistics, Northern Arizona University,
801 S Osborne Drive,
Flagstaff, AZ 86011, USA.}
\email{michele.torielli@nau.edu}


\date{\today}
\maketitle



\section{Introduction}
Recently there has been an increasing interest in the weak and strong Lefschetz properties and, in particular, in understanding which homogeneous ideals fail the Lefschetz properties. This interest have been mainly motivated by the fact that many authors proved, using different techniques, that any Artinian monomial complete intersection has the SLP, see for example \cite{Stanley80} and \cite{Reid91}. In addition, this interested have been strongly boosted by the fact that the WLP and SLP are strongly connected to many topics in algebraic geometry, commutative algebra and combinatorics.
Moreover, the study of Lefschetz properties has already been linked to the theory of free hyperplane arrangements and to Terao's conjecture. See for example \cite{cookline2018}, \cite{di2014singular}, \cite{palezzato2020lefschetz} and \cite{palezzato2020klef}.
We refer to \cite{harima2013lefschetz} for an overview of the Lefschetz properties in the Artinian case, and to \cite{migliore2013survey} for several open questions in the area.

The goal of this paper is to continue and extend the study of the Lefschetz properties for non-Artinian algebras the second and third authors started in \cite{palezzato2020lefschetz} and \cite{palezzato2020klef}.

This paper is organized as follows. In Section 2, we recall the notions of generic initial ideals, and of  weak and strong Lefschetz properties and their relations. In Section 3, we first consider the quotients whose dimension is bigger or equal than 2, providing equivalent conditions to ensure the strong Lefschetz property (SLP) (see Theorem \ref{theo:bigequivbigdimSLP}). Furthermore, we prove the equivalence of the strong and weak Lefschetz properties in this case. After that, we focus on the 1-dimensional case, proving that the quotient ring always satisfies the weak Lefschetz property (WLP) (see Theorem \ref{theo:WLP3var1dim}).\\
Finally, in Section 4, we apply our results to the particular case of hyperplane arrangements. In particular, we show that:
\begin{itemize}
    \item WLP holds for any such algebra in the 3-space (see Theorem \ref{theo:WLParr3var}) and we determine in which degree it fails for $n$-spaces with $n>3$ (see Theorem \ref{theo-WLPfail}).
    \item SLP holds for any algebra given by plus-one generated arrangement (see Theorem \ref{theo-plusone}).
\end{itemize}

\section{Preliminaries}

In this section, we will introduce the required definition, recall some known result and present some new ones that will be used throughout this work.


\subsection{Generic initial ideals}
Let $\KK$ be a field of characteristic $0$ and consider the polynomial ring $S=\KK[x_0,\dots,x_n]$. 
 
A monomial ideal $I$ of $S$ is said to be \textit{strongly stable} if, for every power-product $t \in I$ and
every $i,j$ such that $1\le i<j\le l$ and $x_j|t$, the power-product $x_i\cdot t/x_j\in I$.


Consider a term ordering $\sigma$ on $S$ and $f$ a non-zero polynomial in $S$ and take 
$$\LT_\sigma(f)= \max_\sigma\{{\rm Supp}(f)\}$$ 
being ${\rm Supp}(f)$ the set of all power-products appearing with non-zero coefficient in $f$. Given an ideal $I$ in $S$, the 
\textit{leading term ideal} of $I$ (known also as the \textit{initial ideal} of $I$), denoted by $\LT_\sigma(I)$, is defined as generated by
 $$\{\LT_\sigma(f)~|~f \in I\setminus\{0\}\}.$$

Galligo proves in \cite{galligo1974propos} the following result, which relates the two notions we have recalled.

\begin{Theorem}\label{thm:Galligo}
Let $I$ be a homogeneous ideal of $S$, with $\sigma$ a term ordering such that
$x_0>_\sigma x_2 >_\sigma \dots >_\sigma x_n$. Then there exists a Zariski
open set $U\subseteq\GL(l)$ and a strongly stable ideal $J$ such that for each  $g\in U$, $\LT_\sigma(g(I)) = J$.
\end{Theorem}

The strongly stable ideal $J$ given in  the previous theorem is called the \textit{generic initial ideal} with respect to $\sigma$ of $I$ and it is denoted by $\gin_\sigma(I)$.  
In particular, when $\sigma$ is the degree reverse lexicographic order, the ideal $\gin_\sigma(I)$ is simply denoted by $\rgin(I)$.

It is possible to deduce much information of an ideal from its generic initial ideal. In particular, given an homogeneous ideal $I$, we are interested in the following three aspects.

\begin{Remark}\label{rem:rginsameHFasideal} 
Let $I$ be a homogeneous ideal of $S$, then:
\begin{description}
\item[a)] The Hilbert function of $S/I$ coincides with the one of $S/\rgin(I)$.
\item [b)] Denote by $\reg(I)$ the Castelnuovo-Mumford regularity of $I$. Then $\reg(I)=\reg(\rgin(I))$. 
Moreover, if $I$ is a strongly stable ideal, then $\reg(I)$ is the highest degree of a minimal generator of $I$ (see \cite{bayer1987criterion}).
\item[c)] Let $I^{\sat}$ denote the saturation of the ideal $I$. Then $\rgin(I^{\sat})=\rgin(I)_{x_n\to0}$. This shows that if $I$ is saturated, then $\rgin(I)$ has no minimal generators involving $x_n$ (see \cite{bayer1987criterion}).
\end{description}
\end{Remark}

Finally, it can be proven that the saturation of an ideal ``commutes'' with the generic initial ideal construction.
\begin{Lemma}\label{lemma:cginsatcommute} Let $I$ be a homogeneous ideal of $S$. Then $\rgin(I^{\sat})=\rgin(I)^{\sat}$.
\end{Lemma}
\begin{proof} Let $t$ be a minimal generator of $\rgin(I^{\sat})$. If $t\in\rgin(I)$, then clearly $t\in \rgin(I)^{\sat}$.\\
If $t\notin\rgin(I)$, then there exists $\alpha\ge1$ such that $tx_n^\alpha\in\rgin(I)$. Since $\rgin(I)$ is a strongly stable ideal, then $tx_i^\alpha\in\rgin(I)$ for all $i=0,\dots, n$, and hence $t\in\rgin(I)^{\sat}$.\\ 
This shows that $\rgin(I^{\sat})\subseteq\rgin(I)^{\sat}$.

On the other hand let $t$ be a minimal generator of $\rgin(I)^{\sat}$. By definition, for each $i=0,\dots, n$, there exists $\alpha_i\ge0$ such that $tx_i^{\alpha_i}\in\rgin(I)$. In particular, $tx_n^{\alpha_n}\in\rgin(I)$ and hence $t\in\rgin(I^{\sat})$. This shows that $\rgin(I^{\sat})\supseteq\rgin(I)^{\sat}$.
\end{proof}

\subsection{Lefschetz properties}
Let $R$ be a graded ring over $\KK$, and $R = \bigoplus_{i\geq 0} R_i$ its decomposition into homogeneous
components with $\dim_\KK (R_i) < \infty$.
\begin{enumerate}
\item The graded ring $R$ is said to have the \textit{weak Lefschetz property (WLP)}, 
if there exists an element $\ell \in R_1$ such that
the multiplication map
\begin{align*}
\times \ell \colon R_i &\rightarrow R_{i+1}\\ 
f &\mapsto \ell \cdot f
\end{align*}
has full-rank for every $i \geq 0$. In this case, $\ell$ is called a \textit{weak Lefschetz element}.
\item The graded ring $R$ is said to have the \textit{strong Lefschetz property (SLP)}, 
if there exists an element $\ell \in R_1$ such that
the multiplication map
\begin{align*}
\times \ell^s \colon R_i &\rightarrow R_{i+s}\\ 
f &\mapsto \ell^sf
\end{align*}
has full-rank for every $i \geq 0$ and $s\geq 1$. In this case, $\ell$ is called a \textit{strong Lefschetz element}.
\end{enumerate}

In \cite{palezzato2020lefschetz}, the authors showed that to check if a quotient algebra has the SLP or the WLP, it is enough to check the quotient by strongly stable ideals.

\begin{Proposition}\label{prop:ILPginLP}
Let $I$ be a homogeneous ideal of $S$.
Then the graded ring $S/I$ has the SLP (respectively the WLP) if and only if $S/\rgin(I)$ 
has the SLP (respectively the WLP) with Lefschetz element $x_n$.
\end{Proposition}
As an immediate consequence we have the following result.
\begin{Corollary}\label{corol:SLPsaturideal} Let $I$ be a saturated ideal of $S$.
Then the graded ring $S/I$ has the SLP and it has an increasing Hilbert function. 
\end{Corollary}
\begin{proof} By Remark~\ref{rem:rginsameHFasideal}, item (c), $\rgin(I)$ has no minimal generators involving $x_n$. This implies that the map $\times x_n^s\colon (S/\rgin(I))_i \to (S/\rgin(I))_{i+s}$ is injective for every $i \geq 0$ and $s\geq 1$. This implies that $S/\rgin(I)$ 
has the SLP with Lefschetz element $x_n$. By Proposition~\ref{prop:ILPginLP}, $S/I$ has that SLP. The second part of the statement follows from Remark~\ref{rem:rginsameHFasideal} and the fact that the map $\times x_n^s\colon (S/\rgin(I))_i \to (S/\rgin(I))_{i+s}$ is injective for every $i \geq 0$ and $s\geq 1$.
\end{proof}

In principle, in order to understand if a ring has the WLP, one has to check an infinite number of multiplication maps. However, the following result from \cite{palezzato2020klef} shows that if we are interested in the WLP or SLP, we can always reconduct to the Artinian case.   

\begin{Theorem}\label{theo:nonartintoartin}
Let $I$ be a homogeneous ideal of $S$. Then the following facts are equivalent
\begin{enumerate}
\item the graded ring $S/I$ has the SLP (respectively the WLP).
\item the graded Artinian ring $S/J$ has the SLP  (respectively the WLP), where $J=I+(x_0, \dots, x_n)^{\reg(I)+1}$.
\end{enumerate}
\end{Theorem}

We will conclude this part introducing some technical results on the Lefschetz properties of a quotient ring, passing through the saturated ideal.\\

Let $I$ be a homogeneous ideal of $S$. Consider the following short exact sequence of standard graded $S$-modules
\begin{equation}\label{eq:sessatnonsat}
0\to I^{\sat}/I \to S/I\to S/I^{\sat}\to 0. 
\end{equation}

Let $\ell\in S_1$ be a general linear form and $i\ge 0$. From the exact sequence \eqref{eq:sessatnonsat} we obtain the following commutative diagram
\begin{equation}\label{eq:commutdiagrammultlinearform}
\xymatrix{ 0 \ar[r] & (I^{\sat}/I)_i \ar[r]^{\iota_i} \ar[d]^{\times \ell} & (S/I)_i \ar[r]^{\rho_i} \ar[d]^{\times \ell} & (S/I^{\sat})_i \ar[r] \ar[d]^{\times \ell} & 0 \\
0 \ar[r] & (I^{\sat}/I)_{i+1} \ar[r]^{\iota_{i+1}} & (S/I)_{i+1} \ar[r]^{\rho_{i+1}} & (S/I^{\sat})_{i+1} \ar[r]  & 0}
\end{equation}

In the subsequent results, we will describe how the injectivity and the surjectivity of the multiplication map are preserved when relating the quotient ring given by $I$ and its saturation.

\begin{Lemma}\label{theo:relationinject} The multiplication map $\times\ell \colon (S/I)_i\to (S/I)_{i+1}$ is injective if and only if the multiplication map $\times\ell \colon (I^{\sat}/I)_i\to (I^{\sat}/I)_{i+1}$ is injective.
\end{Lemma}
\begin{proof} 
Directly from Corollary~\ref{corol:SLPsaturideal}, we have that the multiplication map $$\times\ell \colon (S/I^{\sat})_i\to (S/I^{\sat})_{i+1}$$ is injective for any $i\ge 0$.
Finally, because of the commutativity of diagram~\eqref{eq:commutdiagrammultlinearform}, we obtain as well that the multiplication map $$\times\ell \colon (S/I)_i\to (S/I)_{i+1}$$ is injective if and only if the multiplication map $$\times\ell \colon (I^{\sat}/I)_i\to (I^{\sat}/I)_{i+1}$$ is also injective.

\end{proof}

\begin{Lemma}\label{theo:relationsurj} The multiplication map $\times\ell \colon (S/I)_i\to (S/I)_{i+1}$ is surjective if and only if the maps $\times\ell \colon (I^{\sat}/I)_i\to (I^{\sat}/I)_{i+1}$ and $\times\ell \colon (S/I^{\sat})_i\to (S/I^{\sat})_{i+1}$ are surjective.
\end{Lemma}
\begin{proof} 
Directly form the commutativity of diagram~\eqref{eq:commutdiagrammultlinearform}, we have that if the multiplication maps defined by 
\begin{equation}\label{surj-sat}
\times\ell \colon (I^{\sat}/I)_i\to (I^{\sat}/I)_{i+1}
\end{equation}
and
\begin{equation}\label{surj-satquotient}
\times\ell \colon (S/I^{\sat})_i\to (S/I^{\sat})_{i+1}
\end{equation}
are surjective, then 
\begin{equation}\label{surj-quotient}
\times\ell \colon (S/I)_i\to (S/I)_{i+1}
\end{equation}
is surjective as well.\\ 

Let us now assume that the multiplication given in (\ref{surj-quotient}) is surjective.
Again by the commutativity of diagram~\eqref{eq:commutdiagrammultlinearform}, we obtain that the surjectivity of (\ref{surj-satquotient}).\\
Recall that, by Proposition~\ref{prop:ILPginLP}, the multiplication map in diagram (\ref{surj-quotient}) is surjective if and only if the
\begin{equation}\label{surj-ginquotient}
\times x_n \colon (S/\rgin(I))_i\to (S/\rgin(I))_{i+1}
\end{equation}
is surjective as well.\\
Analogously, it is possible to prove that the multiplication in (\ref{surj-sat}) is surjective if and only if the map 
\begin{equation}\label{surj-ginsat}
\times x_n \colon (\rgin(I^{\sat})/\rgin(I))_i\to (\rgin(I^{\sat})/\rgin(I))_{i+1}
\end{equation}
 is also surjective. Therefore, it is sufficient to prove that the surjectivity of the multiplication map described in diagram (\ref{surj-ginquotient}) implies the surjectivity of the multiplication in diagram (\ref{surj-ginsat}).
 Suppose the first one to be surjective. Then $(x_0,\dots, x_{n-1})^{i+1}\subseteq\rgin(I)_{i+1}$ and hence $(\rgin(I^{\sat})/\rgin(I))_{i+1}$ is generated only by monomials divisible by $x_n$. However, by Remark~\ref{rem:rginsameHFasideal}, item c), the ideal $\rgin(I^{\sat})$ has no minimal generators divisible by $x_n$ and, therefore, the multiplication described in diagram (\ref{surj-ginsat}) is also surjective.


\end{proof}

\section{Higher dimensional ideals}

In this paper, we are interested in the study of non-Artinian  quotient rings, i.e. quotients of dimension greater or equal to $1$. To achieve our goal, we will first analyze the case of quotient rings with dimension bigger than $2$ and then the case of quotient rings with dimension $1$ that are almost complete intersections.

\subsection{Dimension bigger than $2$} Let us first consider the case of ideals such that the dimension of the quotient is bigger than $2$.
\begin{Theorem}\label{theo:bigequivbigdimSLP} Let $n\ge2$ and $I$ a homogeneous ideal of $S=\KK[x_0,\dots,x_n]$. Assume that $\dim(S/I)\ge 2$. Then the following facts are equivalent
\begin{enumerate}
\item $S/I$ has the SLP,
\item $\projdim(S/I)\le n$,
\item $\depth(S/I)\ge 1$,
\item $\mathfrak{m}\notin\Ass(S/I)$,
\item $I$ is saturated.
\end{enumerate}
Moreover, if $S$ has standard grading, then the previous properties are equivalent also to the following:
\begin{enumerate}
\item[(a)] $S/I$ has the WLP.
\end{enumerate}
\end{Theorem}
\begin{proof}
Since facts (2), (3), (4) and (5) are known to be equivalent, we just need to prove that (1) and (2) are equivalent.

Assume that $S/I$ has the SLP. By Proposition~\ref{prop:ILPginLP}, $S/\rgin(I)$ has the SLP with Lefschetz element $x_n$. By \cite[Proposition 2.12]{palezzato2020lefschetz}, the Hilbert function of $S/\rgin(I)$ is unimodal. However, since $\dim(S/\rgin(I))=\dim(S/I)\ge 2$, then the tail of the Hilbert function is strictly increasing, and hence the Hilbert function of $S/\rgin(I)$ is an increasing function. This implies that the multiplication map $\times x_n\colon (S/\rgin(I))_i\to (S/\rgin(I))_{i+1}$ is injective for any $i\ge0$. As a consequence $\rgin(I)$ has no minimal generators divisible by $x_n$, and hence $\projdim(S/I)=\projdim(S/\rgin(I))\le n$. This shows that (1) implies (2).

Assume now that $\projdim(S/I)\le n$. By \cite[Theorem A1.9]{eisenbudsyzygy}, $H^0_{\mathfrak{m}}(S/I)\cong \Ext^{n+1}(S/I,S(-n-1))$.  Since $\projdim(S/I)\le n$, we have $\Ext^{n+1}(S/I,S(-n-1))=0$, and hence $I^{\sat}/I\cong H^0_{\mathfrak{m}}(S/I)=0$. This implies that $I$ is a saturated ideal. We then conclude by Corollary~\ref{corol:SLPsaturideal}, and hence (2) implies (1).

Assume now that $S$ has standard grading. By Definition of SLP, clearly (1) implies (a). Assume that $S/I$ has the WLP. By Proposition~\ref{prop:ILPginLP}, $S/\rgin(I)$ has the WLP with Lefschetz element $x_n$. By \cite[Proposition 2.10]{palezzato2020lefschetz}, the Hilbert function of $S/\rgin(I)$ is unimodal. Using the same proof as before, we can prove that (a) implies (2).
\end{proof}

\begin{Remark} Notice that if in Theorem~\ref{theo:bigequivbigdimSLP} we do not assume that $\dim(S/I)\ge 2$, then we still obtain that (1) implies (a), and (2) implies (1) and (a) but the opposite implications are false. In fact, it is enough to consider the ideal $I=\ideal{x^2,xy,y^2,xz^2}$ in $S=\KK[x,y,z]$. Then $S/I$ has the WLP with Lefschetz element $z$ but it does not have the SLP. On the other hand, if we consider the ideal $I=\ideal{x^3,x^2y,,xy^3,y^4,xy^2z}$ in $S=\KK[x,y,z]$. Then $S/I$ has the SLP with Lefschetz element $z$ but $\projdim(S/I)=3$.
\end{Remark}

\subsection{Dimension $1$ almost complete intersections}
Let $I$ be the ideal generated by the homogeneous polynomials $f_0, f_1, f_2$ of degree respectively $d_0, d_1, d_2$ in $S=\KK[x_0,x_1,x_2]$. Assume in addition that $\dim(S/I)=1$ or equivalently that $V(I)\subset \PP^2$ is 0-dimensional (the case $\dim(S/I)=0$ is already studied in \cite{Brenner07}). Notice that under our assumptions, we are studying dimension $1$ almost complete intersections.\\

The goal of this section is to prove the following result.
\begin{Theorem}\label{thm-main}
The quotient ring \label{theo:WLP3var1dim} $S/I$ has the WLP.
\end{Theorem}

As observed in \cite{dimcapoplefs}, we can describe the Hilbert-Poincar\'e series of $S/I^{\sat}$ as
$$HS_{S/I^{\sat}}(t)=F/(1-t),$$
where $F\in\mathbb{Z}[t]$ and $F(1)\ne0$. Moreover, $\deg(F)$ is the minimal degree $p$ for which $\dim_{\KK}((S/I^{\sat})_k)=F(1)$ for all $k\ge p$.
\begin{Remark} Since $S/I^{\sat}$ is Cohen-Macaulay of dimension $1$, we have that $\deg(F)=\reg(I^{\sat})-1$.
\end{Remark}
We will now study the involved multiplication maps. 
\begin{Lemma}\label{lemma:rightcolumn} Let $\ell\in S_1$ be a general linear form. Then the multiplication map 
$$\times\ell \colon (S/I^{\sat})_i\to (S/I^{\sat})_{i+1}$$ 
is injective for any $0\le i\le \deg(F)-1$ and it is bijective for any $i\ge\deg(F)$.
\end{Lemma}
\begin{proof} By Corollary~\ref{corol:SLPsaturideal}, the multiplication map is injective for all $i\ge0$. However, since $\dim_{\KK}((S/I^{\sat})_i)$ is constant for any $i\ge\deg(F)$, this implies that the map is bijective for any $i\ge\deg(F)$.
\end{proof}

In order to study the multiplication map defined by 
$$
\times\ell \colon (I^{\sat}/I)_i\to (I^{\sat}/I)_{i+1}
$$
we will relate it to the cohomology of the syzygy vector bundle of the ideal sheaf.

Indeed, we can consider the short exact sequence
\begin{equation}\label{eqdefsyzygy}
0\to\T\to\bigoplus_{i=0}^2\OO_{\PP^2}(-d_i)\stackrel{M}{\to}\mathcal{I}_\Gamma\to0,
\end{equation}
where  $A = [f_0 \:\: f_1 \:\: f_2 ]$ and $\Gamma$ denotes the zero locus of the ideal $I$. Notice that, under our assumptions, $\T$ is a vector bundle.
Let us denote
$$m(I)=\min\{q\in\mathbb{Z}~|~\T(m(I))\ne0 \text{ and } \T(m(I)-1)=0\}.$$
Observe that $m(I)$ coincides with the minimal degree of a syzygy of the matrix $M$.
Indeed, an element $s\in H^0(\T(m(I)))$ defines the following commutative diagram
$$\xymatrix{ \OO_{\PP^2}(-m(I)) \ar[d]^-{s} \ar[rd]^-{[g_0,g_1,g_2]^t} &  & & \\
\T \ar[r] & \bigoplus_{i=0}^2\OO_{\PP^2}(-d_i) \ar[r]^-{M} & \mathcal{I}_\Gamma \ar[r] & 0}
$$
with the relation $\sum_{i=0}^2f_ig_i=0$ and such that $\deg(g_i)=m(I)-d_i$ for $i=0,1,2$.

As observed in \cite{sernesilocal}, this allows to have an identification
\begin{equation}\label{eqidenlocalh1syz}
(I^{\sat}/I)_i\cong (H^0_{\mathfrak{m}}(S/I))_i\cong H^1(\T(i)),
\end{equation}
that leads to the following equality (see \cite[Corollary 1.4]{dimcapoplefs})
\begin{equation}\label{eq:degFmdr}
\deg(F)= d_0+d_1+d_2-m(I)-2.
\end{equation}

\begin{Proposition}\label{prop-bounddeg}
Let $\ell\in S_1$ be a general linear form. Consider the multiplication map 
\begin{equation}\label{cohom-mult}
\varphi_i \colon H^1(\T(i))\stackrel{\times\ell}{\to} H^1(\T(i+1))
\end{equation}.
Then:
\begin{itemize}
    \item If $\T$ is unstable, then $\varphi_i$ is injective if $i<d_0+d_1+d_2 - m(I) -1$ and it is surjective if $i>m(I)-3$; moreover, these bounds are sharp.
    \item If $\T$ is either stable or semistable, then $\varphi_i$ is injective if $i < \frac{d_0+d_1+d_2-2}{2}$, for $d_0+d_1+d_2$ even, and if $i < \frac{d_0+d_1+d_2-3}{2}$, for $d_0+d_1+d_2$ odd. Moreover, it is sujective if $i > \frac{d_0+d_1+d_2-6}{2}$, for $d_0+d_1+d_2$ even, and if $i > \frac{d_0+d_1+d_2-5}{2}$, for $d_0+d_1+d_2$ odd.
\end{itemize}
\end{Proposition}

\begin{proof} If $\T$ is a direct sum of line bundles, then the statement is true because of the vanishing of $H^1(\T(i))$, for any $i\in \mathbb{Z}$. Therefore, from now on, we assume $\T$ to be indecomposable.\\
Let us study first the injectivity of the map (\ref{cohom-mult}).

Firstly, we consider the case that $\T$ is unstable. This implies that
$$-m(I)>(-d_0-d_1-d_2)/2.$$
Consider a generic line $L\subseteq\PP^2$ defined by the linear form $\ell$. Moreover, consider the short exact sequence
$$0\to\OO_{\PP^2}(-m(I))\to\T\to\mathcal{I}_Y(m(I)-d_0-d_1-d_2)\to 0 $$
and its restriction to the line $L$
$$0\to\OO_{L}(-m(I))\to\T_{|L}\to\OO_L(m(I)-d_0-d_1-d_2)\to 0. $$
Since $d_0+d_1+d_2-2m(I)>0$, this implies that
$$\Ext^1(\OO_L(m(I)-d_0-d_1-d_2),\OO_L(-m(I)))\cong H^1(\OO_L(d_0+d_1+d_2-2m(I)))=0.$$
and, therefore, we obtain the splitting, for the generic line $L$,
$$\T_{|L}\cong\OO_L(-m(I))\oplus\OO_L(m(I)-d_0-d_1-d_2).$$
Moreover, being $\T$ indecomposable, we have $Y\ne\emptyset$ and therefore $H^0(\mathcal{I}_Y(i))=0$ for all $i\le0$. This implies that
$H^0(\T(i))\cong H^0(\OO_{\PP^2}(i-m(I)))$ for all $i\le d_0+d_1+d_2-m(I)$.\\
Consider the short exact sequence
$$0\to \T(i)\to T(i+1)\to\OO_L(i+1-m(I))\oplus \OO_L(i+1+m(I)-d_0-d_1-d_2)\to0.$$
and its derived long exact sequence in cohomology
$$0\to H^0(\T(i))\to H^0(T(i+1))\to H^0(\OO_L(i+1-m(I)))\oplus H^0(\OO_L(i+1+m(I)-d_0-d_1-d_2)) $$
$$\to H^1(\T(i))\to H^1(T(i+1))\to \cdots$$
If $i<d_0+d_1+d_2-m(I)-1$, then we have the following vanishing 
$$H^0(\OO_L(i+1+m(I)-d_0-d_1-d_2))=0,$$ and dimension equivalence
$$h^0(\T(i+1))-h^0(\T(i))=h^0(\OO_{\PP^2}(i+1-m(I)))-h^0(\OO_{\PP^2}(i-m(I)))=h^0(\OO_{L}(i+1-m(I))).$$
Therefore, we have that $\varphi_i$ is injective for all $i<d_0+d_1+d_2-m(I)-1$.\\
Observe that, being $\T$ indecomposable, $\varphi_{(d_0+d_1+d_2-m(I)-1)}$ cannot be injective having a $1$-dimensional kernel, hence the bound is sharp.

The map $\varphi_i$ is surjective if and only if we have the short exact sequence in cohomology
$$
0 \to H^1(\OO_L(i+1-m(I)))\oplus H^1(\OO_L(i+1+m(I)-d_0-d_1-d_2)) \to  H^2(\T(i))\to H^2(T(i+1)) \to 0,
$$
which, by Serre duality, is equivalent to the following one
$$
\begin{array}{rcl}
&H^1(\OO_L(m(I)-i-3))\\
0 \to H^0(\T(d_0+d_1+d_2-i-4)) \to H^0(\T(d_0+d_1+d_2-i-3)) \to &\oplus & \to 0\\
& H^1(\OO_L(-i-3-m(I)+d_0+d_1+d_2)).
\end{array}
$$
Furthermore, this is equivalent to ask for the map $\phi_{d_0+d_1+d_2-i-4}$ to be injective. As seen before, this happens if and only if $d_0+d_1+d_2-i-4 \leq \deg(F) = d_0+d_1+d_2-m(I)-2.$ Therefore, we can state that $\varphi_i$ is surjective if and only if $i > m(I) - 3$.

Suppose now that $\T$ is either stable or semistable. This implies that 
$$-m(I)\le(-d_0-d_1-d_2)/2.$$
If $d_0+d_1+d_2$ is even, due to Grauert-M\"ulich theorem (see \cite{okonek} for a reference), we have
$$\T_{|L}\cong\OO_{L}((-d_0-d_1-d_2)/2)^2.$$
Considering the short exact sequence
$$0\to \T(i)\to T(i+1)\to \OO_{L}(i+1+(-d_0-d_1-d_2)/2)^2\to0,$$
we have that 
$$
H^0(\OO_L(i+1+(-d_0-d_1-d_2)/2) = 0 \:\mbox{ and therefore }\: \varphi_i \colon H^1(\T(i))\to H^1(T(i+1))
$$
is injective for all $i<(d_0+d_1+d_2-2)/2$.\\
Again by Serre duality applied to the cohomology sequence, we have that $\varphi_i$ is surjective if $i > (d_0+d_1+d_2-6)/2.$

If $d_0+d_1+d_2$ is odd, similarly to the even case, we have the following generic splitting type
$$\T_{|L}\cong\OO_{L}((-d_0-d_1-d_2-1)/2)\oplus\OO_{L}((-d_0-d_1-d_2+1)/2).$$
which leads to the short exact sequence
$$0\to \T(i)\to T(i+1)\to \OO_{L}(i+1+(-d_0-d_1-d_2-1)/2)\oplus\OO_{L}(i+1+(-d_0-d_1-d_2+1)/2)\to0.$$
This implies that the map $\varphi$ is injective for all $i<(d_0+d_1+d_2-3)/2$.\\ 
Once again, by Serre duality applied to the cohomology sequence, we have that $\varphi_i$ is surjective if $i \geq (d_0+d_1+d_2-3)/2.$
\end{proof}

To prove the main result of this section, we need the following result, which is a direct consequence of the previous proposition.

\begin{Corollary}\label{theo:leftcolum3var}
Let $\ell\in S_1$ be a general linear form. Then the multiplication map 
\begin{equation}
\varphi_i \colon H^1(\T(i))\stackrel{\times\ell}{\to} H^1(\T(i+1))
\end{equation}
 is injective for any $i\le \deg(F)-1$ and it is surjective for any $i\ge\deg(F)$.
\end{Corollary}
\begin{proof}
If $\T$ is not stable, we have that $\deg(F)+1\le (d_0+d_1+d_2-2)/2$ and $\deg(F) \geq m(I) - 2$, which imply the injectivity and surjectivity of $\varphi_i$ if, respectively, $i \leq \deg(F)$ and $i \geq \deg(F)$, due to Proposition \ref{prop-bounddeg}.

  If $\T$ is either semistable or stable, then $\deg(F)\le (d_0+d_1+d_2-4)/2$, which implies, again by the Proposition \ref{prop-bounddeg}, that $\varphi$ is injective for all $i\le \deg(F)-1$.
Finally, recall that the map $\times\ell \colon H^1(\T(i))\to H^1(T(i+1))$ is surjective for all $i\ge \deg(F)$ follows from \cite[Theorem 4.1]{dimcapoplefs}.\\
\end{proof}

\begin{proof} [Proof of Theorem \ref{thm-main}] 
By Corollary~\ref{theo:leftcolum3var} and the identification described in diagram~\eqref{eqidenlocalh1syz}, the multiplication map 
$$\times\ell \colon (I^{\sat}/I)_i\to (I^{\sat}/I)_{i+1}$$ 
is injective for any $i\le \deg(F)-1$ and surjective for all $i\ge \deg(F)$.\\ 
Combined with Lemma~\ref{theo:relationinject} for the injectivity and with Lemma~\ref{theo:relationsurj} and Lemma~\ref{lemma:rightcolumn} for the surjectivity, this implies that the multiplication map $$\times\ell \colon (S/I)_i\to (S/I)_{i+1}$$ is also injective for any $i\le \deg(F)-1$ and surjective for all $i\ge \deg(F)$. This proves the result.
\end{proof}

As an immediate consequence, we obtain the following result.

\begin{Corollary}\label{corol:generalizconjBPTginfree} Consider $p_0=\min\{p~|~x_1^{p}\in\rgin(I)\}$.
If $\rgin(I)$ has a minimal generator $T$ that involves $x_2$, then $\deg(T)\ge p_0$.
\end{Corollary}
\begin{proof} Suppose that there exists a minimal generator $t$ of $\rgin(I)$, with $\deg(t) = g < p_0$ and $x_2 \mid t$.
By Theorem~\ref{theo:WLP3var1dim}, the quotient ring $S/I$ has the WLP and hence $S/\rgin(I)$ has the WLP as well with Lefschetz element $x_2$. This implies that the multiplication map 
$$\times x_2\colon (S/\rgin(I))_{g-1}\to (S/\rgin(I))_{g}$$ has maximal rank.
However, this map cannot be injective since $\frac{t}{x_2}$ belongs to its kernel and it cannot be surjective since $x_1^{g}$ is a non-zero element of the cokernel. This leads to a contradiction with the existence of the element $t$.
\end{proof}

\section{Hyperplane arrangements and Lefschetz properties}\label{sec:arr}



The study of the Lefschetz properties has been lately of interest when considering, as an ideal, the Jacobian ideal associated to an hyperplane arrangement. After recalling the necessary definitions and properties, we will describe in which cases the Lefschetz properties can be assured, due to the results in the previous section. For more details on hyperplane arrangements, see \cite{orlterao}. 

A finite set of affine hyperplanes $\A =\{H_1, \dots, H_d\}$ in $\KK^{n+1}$ 
is called a \textit{hyperplane arrangement}. For each hyperplane $H_i$ we fix a defining linear polynomial $\alpha_i\in S$ such that $H_i = \alpha_i^{-1}(0)$, 
and let $Q(\A)=\prod_{i=1}^d\alpha_i$. An arrangement $\A$ is called \textit{central} if each $H_i$ contains the origin of $\KK^{n+1}$. 
In this case, each $\alpha_i\in S$ is a linear homogeneous polynomial, and hence $Q(\A)$ is homogeneous of degree $d$.
Moreover, an arrangement $\A$ is called \textit{essential} if there are $H_{i_1},\dots, H_{i_k}\in\A$ such that $\codim(H_{i_1}\cap\cdots\cap H_{i_k})=n+1$.


We denote by $\Der_{\KK^{n+1}} =\{\sum_{i=0}^{n} f_i\partial_{x_i}~|~f_i\in S\}$ the $S$-module of \textit{polynomial vector fields} on $\KK^{n+1}$ (or $S$-derivations). 
Let $\delta =  \sum_{i=0}^{n} f_i\partial_{x_i}\in \Der_{\KK^{n+1}}$. Then $\delta$ is  said to be \textit{homogeneous of polynomial degree} $q$ if $f_0, \dots, f_{n}$ are homogeneous polynomials of degree~$q$ in $S$. 
In this case, we write $\pdeg(\delta) = q$.

\begin{Definition} 
Let $\A$ be a central arrangement in $\KK^{n+1}$. Define the \textit{module of vector fields logarithmic tangent} to $\A$ (or logarithmic vector fields) by
$$D(\A) = \{\delta\in \Der_{\KK^{n+1}}~|~ \delta(\alpha_i) \in \ideal{\alpha_i} S,\: i=1,\ldots,d\}.$$
\end{Definition}

The module $D(\A)$ is a graded $S$-module and $D(\A)= \{\delta\in \Der_{\KK^{n+1}}~|~ \delta(Q(\A)) \in \ideal{Q(\A)} S\}$. 

\begin{Definition} 
A central arrangement $\A$ in $\KK^{n+1}$ is said to be \textit{free with exponents $(e_1,\dots,e_{n+1})$} 
if and only if $D(\A)$ is a free $S$-module and there exists a basis $\delta_1,\dots,\delta_{n+1} \in D(\A)$ 
such that $\pdeg(\delta_i) = e_i$, or equivalently $D(\A)\cong\bigoplus_{i=1}^{n+1}S(-e_i)$.
\end{Definition}



Given an arrangement $\A$ in $\KK^l$, the \textit{Jacobian ideal} $J(\A)$ of $\A$
is the ideal of $S$ generated by $Q(\A)$ and all its partial derivatives.
The Jacobian ideal has a central role in the study of free arrangements.
In fact, we can characterize freeness by looking at $J(\A)$ via the Terao's criterion.
Terao described this result for characteristic $0$, but the statement holds true for any characteristic as shown in \cite{palezzato2018free}.

\begin{Theorem}[\cite{terao1980arrangementsI}]\label{theo:freCMcod2} 
A central arrangement $\A$ in $\KK^{n+1}$ is free if and only if $S/J(\A)$ is $0$ or $(n-1)$-dimensional Cohen--Macaulay.
\end{Theorem}

In \cite{Gin-freearr}, the authors connected the study of generic initial ideals to the one of arrangements, obtaining a new characterization of freeness via the generic initial ideal of the Jacobian ideal.

\begin{Proposition}[\cite{Gin-freearr}]\label{prop:shapergin}
Let $\A =\{H_1, \dots, H_d\}$ be a central arrangement in $\KK^{n+1}$. 
Then $\rgin(J(\A))$ coincides with $S$ or 
its minimal generators include $x_0^{d-1}$, some positive power of
$x_1$, and no monomials only in $x_2,\dots, x_{n}$. 
\end{Proposition}


\begin{Theorem}[\cite{Gin-freearr}]\label{theo:firstequivfregin}
Let $\A =\{H_1, \dots, H_d\}$ be a central arrangement in $\KK^{n+1}$. 
Then 
$\A$ is free if and only if
$\rgin(J(\A))$ coincides with $S$ or it is minimally generated by
$$x_0^{d-1}, x_0^{d-2}x_1^{\lambda_1}, \dots, x_1^{\lambda_{d-1}}$$ 
with $1\le\lambda_1<\lambda_2<\cdots<\lambda_{d-1}$ and $\lambda_{i{+}1}-\lambda_i= 1$ or $2$.
\end{Theorem}


The following Conjecture appeared in \cite{Gin-freearr}.

\begin{Conjecture}\label{conj:generatZ}
Let $\A$ be a central arrangement in $\KK^{n+1}$, 
and consider $p_0=\min\{p~|~x_1^{p}\in\rgin(J(\A))\}$.
If $\rgin(J(\A))$ has a minimal generator $t$ that involves $x_2$, then $\deg(t)\ge p_0$.
\end{Conjecture}



In \cite{palezzato2020lefschetz} and then in \cite{palezzato2020klef}, the authors studied the connection between the Jacobian algebra $S/J(\A)$ of an arrangement $\A$ and the Lefschetz properties obtaining the following results.

\begin{Proposition}{\cite{palezzato2020lefschetz}}\label{prop:l=2SLP}
Let $\A$ be a central arrangement in $\KK^2$. Then $S/J(\A)$ has the SLP.
\end{Proposition}




\begin{Theorem}{\cite{palezzato2020lefschetz}}\label{thm:FreeSLP}
Let $\A$ be a free arrangement in $\KK^{n+1}$. Then $S/J(\A)$ has the SLP.
\end{Theorem}

Directly from Theorem~\ref{theo:WLP3var1dim}, we obtain the following result.
\begin{Theorem}\label{theo:WLParr3var} Let $\A$ be a central and essential arrangement in $\KK^3$. Then $S/J(\A)$ has the WLP.
\end{Theorem}

Notice that there are arrangements in $\KK^3$ such that their Jacobian algebra does not have the SLP.

\begin{Example} Let $\A$ be the arrangement in $\RR^3$ with defining polynomial $Q(\A)=xz(x-z)(x-y)(y-z)(y-2z)(y-3z)(y-4z)$.
In this case we have that $\HF(S/\rgin(J(\A)),8)=\HF(S/\rgin(J(\A)),11)=36$ and that $x^2y^6z^3$ is a minimal generator of $\rgin(J(\A))$.
This shows that the multiplication map $\times z^3\colon (S/\rgin(J(\A)))_8 \to (S/\rgin(J(\A)))_{10}$ does not have maximal rank, and hence that $z$ is not a Lefschetz element for $S/\rgin(J(\A))$. By \cite[Lemma 2.8]{palezzato2020lefschetz},
$S/\rgin(J(\A))$ does not have the SLP. This implies that by Proposition~\ref{prop:ILPginLP}, also $S/J(\A)$ does not have the SLP. 
\end{Example}

Directly from Corollary~\ref{corol:generalizconjBPTginfree}, we obtain the following result.
\begin{Theorem} Conjecture~\ref{conj:generatZ} holds true for arrangements in $\KK^3$.
\end{Theorem}


\begin{Theorem}\label{theo-WLPfail}
Let $\A$ be a central and essential arrangement in $\KK^{n+1}$ with $d=|\A|$. If $S/J(\A)$ fails the WLP in degree $k$, then $k\ge d$.
\end{Theorem}
\begin{proof}
Since $J(\A)$ is generated in degree $d-1$ then $S/J(\A)$ has the WLP in any degree $k \leq d-2$. 
By generalizing \cite[Theorem 3.2]{mezzetti2013laplace} to the non Artinian case, if $S/J(\A)$ fails the WLP in degree $d-1$, then the partial derivatives of $Q(\A)$ are $\KK$-dependent on every general central hyperplane $H$. This implies that there exists $\delta$ a logarithmic vector field of degree zero in $D(A^{H})$, where $\A^H$ is the restriction of $\A$ to $H$. However, this is impossible since $\A^H$ is essential from the fact that $\A$ is essential.
\end{proof}


In \cite{abe2018plus}, the author generalized the notions of free and nearly free (see \cite{dimca2015nearly}) to central and essential arrangements in any dimension. However, the definition can be given also for non-essential arrangements.


\begin{Definition}\label{def:plusonegen}
Let $\A=\{H_1,\dots,H_d\}$ be an arrangement in $\KK^{n+1}$. We say that $\A$ is \textit{plus-one generated} with \textit{exponents} $\POexp(\A)=(a_1,\dots,a_{n+1})$ and \textit{level} $a$ if $D(\A)$ has a minimal free resolution of the following form
\begin{equation}\label{eq:respogDA}
\xymatrix{0 \ar[r] & S(-a-1) \ar[rr]^-{(\alpha,g_1,\dots,g_{n+1})} && S(-a) {\oplus}(\bigoplus_{i=1}^{n+1} S(-a_i)) \ar[r] & D(\A) \ar[r] & 0 }
\end{equation}
\end{Definition}

\begin{Remark} Let $\A$ be a plus-one generated arrangement in $\KK^{n+1}$ with exponents $\POexp(\A)=(a_1,\dots,a_{n+1})_{\le}$ and level $a$. Since $\A$ is central, then there exists $k\ge2$ such that $(a_1,\dots,a_{n+1})_{\le}=(0,\dots,0,1,a_k,\dots,a_{n+1})_{\le}$. If $\A$ is essential, then $k=2$. If $\A$ is non-essential, then $k\ge3$.
\end{Remark}

Directly from the definition, we can show the following (see \cite[Lemma 3.3]{palezzato2021localiz})

\begin{Lemma}\label{lemma:fromresDAtoJac}
Let $\A=\{H_1,\dots,H_d\}$ be an arrangement in $\KK^{n+1}$. $\A$ is plus-one generated with exponents $\POexp(\A)=(a_1,\dots,a_{n+1})_{\le}$ and level $a$
if and only if $S/J(\A)$ has a minimal free resolution of the form
\begin{equation}\label{eq:respogjac}
0{\to}S(-d-a){\to} S(-d-a+1){\oplus}(\bigoplus_{i=k}^{n+1} S(-d-a_i+1)) {\to} S(-d+1)^{n-k+3}{\to}S.
\end{equation}
Moreover, the map 
$$\partial_3\colon S(-d-a)\to S(-d-a+1)\oplus(\bigoplus_{i=k}^{n+1} S(-d-a_i+1))$$
is defined by a matrix of the form $(\alpha, g_{i_1}, \dots, g_{i_{n-k+2}})$, where $1\le i_1<\cdots<i_{n-k+2}\le n+1$.
Notice that $n-k+3$ coincides with the codimension of the center of $\A$ or, equivalently, the rank of $\A$.
\end{Lemma}

Using Theorem~\ref{theo:bigequivbigdimSLP}, we have the following.
\begin{Theorem}\label{theo-plusone} 
Let $\A$ be a plus-one generated arrangement in $\KK^{n+1}$ with $n\ge3$. Then $S/J(\A)$ has the SLP.
\end{Theorem}






\end{document}